\documentclass[12pt]{amsart}
\usepackage{fullpage}
\usepackage{latexsym}
\usepackage{amssymb}
\usepackage[all]{xy}

\newtheorem{thm}{Theorem}[section]

\newtheorem{cor}[thm]{Corollary}

\newtheorem{lem}[thm]{Lemma}
\newtheorem{que}[thm]{Question}

\theoremstyle{definition}
\newtheorem{rem}[thm]{Remark}
\newtheorem{defn}[thm]{Definition}
\newtheorem{ex}[thm]{Example}

\usepackage{hyperref}

\newcommand{\Q}{\mathbb{Q}}
\newcommand{\R}{\mathbb{R}}
\newcommand{\Z}{\mathbb{Z}}

\newcommand{\SO}{\mathrm{\SO}}

\newcommand\rank{\operatorname{rank}}

\title[Topological complexity, asphericity and connected sums]
{Topological complexity, asphericity and \\connected sums}

\author{Christoforos Neofytidis}
\address{Department of Mathematics, Ohio State University, Columbus, OH 43210, USA}
\email{neofytidis.1@osu.edu}
\date{\today}
\subjclass[2010]{}

\begin{document}

\begin{abstract}
We show that if a closed oriented $n$-manifold $M$ has a non-trivial cohomology class of even degree $k$, whose all pullbacks to products of type $S^1\times N$ vanish, then the topological complexity $\mathrm{TC}(M)$ is at least $6$, if $n$ is odd, and at least $7$ or $9$, if $n$ is even. These bounds extend and improve a result of Mescher and apply for instance to negatively curved manifolds and to connected sums with at least one such summand. In fact, better bounds are obtained due to the non-vanishing of the Gromov norm. As a consequence, in dimension four, we completely determine the topological complexity of these connected sums, namely we show that it is equal to its maximum value nine. Furthermore, we discuss realisation of degree two homology classes by tori, and show how to construct  non-realisable classes out of realisable classes in connected sums. The examples of this paper will quite often be aspherical manifolds whose fundamental groups have trivial center and  connected sums. We thus discuss the possible relation between the maximum topological complexity $2n+1$ and the triviality of the center for aspherical $n$-manifolds and their connected sums.
 \end{abstract}

\maketitle

\section{Introduction}

 The Lusternik-Schnirelmann (LS) category of a topological space $X$, denoted by $\mathrm{cat}(X)$,  is the smallest integer $n$ such that there exist an open covering $\{U_1,...,U_n\}$ of $X$, where each $U_i$ is contractible in $X$. A motion planning algorithm over an open subset $U_i\subset X\times X$ is a continuous map 
\[
s_i\colon U_i\to X^{[0,1]}\colon (x,y)\mapsto s, \ s(0) = x, s(1) = y, \ \text{where} \ (x,y)\in U_i.
\]
The topological complexity $\mathrm{TC}(X)$ of $X$, introduced by Farber~\cite{Fa03}, is the smallest integer $n$ such that there exist a covering $\{U_1,...,U_n\}$ of $X\times X$ with $n$ open sets over which there are motion planning algorithms\footnote{Throughout this paper we use the {\em unreduced} versions of the LS-category and the topological complexity as in~\cite{Fa03}, that is, their value is $+1$ the value of the reduced version which is used in various references of this paper.}. (Note that $\pi_X\circ s_i=id_{U_i}$, for all $i=1,...,n$, where  $\pi_X$ denotes the free path fibration.) Both $\mathrm{cat}(X)$ and $\mathrm{TC}(X)$ can be defined using the Schwarz genus, which makes their relationship transparent. An initial observation (see~\cite[Theorem 5]{Fa03}) is that the LS-category provides the following bounds for the topological complexity
\begin{equation}\label{eq:LS-TC}
\mathrm{cat}(X)\leq\mathrm{TC}(X)\leq\mathrm{cat}(X\times X)\leq 2\mathrm{cat}(X)-1.
\end{equation}
In particular, since $\mathrm{cat}(X)\leq\dim(X)+1$ (\cite[Theorem 1.7]{CLOT03}), we have the following (see~\cite[Theorem 4]{Fa03})
\begin{equation}\label{eq:dim-TC}
\mathrm{TC}(X)\leq2\dim(X)+1.
\end{equation}

An efficient way to obtain lower bounds for the topological complexity is through topological complexity weights (see Definition \ref{d:tcwgt}). Mescher proved that if a closed oriented manifold of dimension at least three contains an {\em atoroidal} cohomology class, that is, a cohomology class of degree two whose all pullbacks to a $2$-torus $T^2$ vanish, then $\mathrm{TC}(M)\geq6$; see~\cite[Theorem 6.1]{Me21}.  Our first goal is to extend and improve Mescher's result as follows:

\begin{thm}\label{t:main}
Let $M$ be a closed oriented manifold of dimension $n\geq3$. Suppose that there is a non-trivial cohomology class $u_a\in H^k(M;\Q)$, where $k$ is even, $1<k<n$,  whose all pullbacks to products $S^1\times N$ vanish for any closed oriented $(k-1)$-manifold $N$. Then
\begin{itemize} 
\item[(a)] $\mathrm{TC}(M)\geq6$, if $n$ is odd;
\item[(b)] $\mathrm{TC}(M)\geq7$, if $n$ is even; 
\item[(c)] $\mathrm{TC}(M)\geq9$, if $n=2k$ and $u_a^2\neq0$.
\end{itemize}
\end{thm}

As usual, we denote by $\omega_X\in H^n(X)$ the cohomological fundamental class of a closed oriented manifold $X$.

Theorem \ref{t:main} applies for instance to negatively curved manifolds and to connected sums containing at least one such a summand. For example, we have the following stronger estimate than Theorem \ref{t:main}(b):

\begin{cor}\label{c:nc&cs}
Let $M$ be a closed oriented negatively curved manifold of dimension $n\geq3$ with $H_k(M;\Q)\neq0$ for some even $k$, $1<k<n$. Then $TC(M\#L)\geq6$ if $n$ is odd, and $\mathrm{TC}(M\#L)\geq9$ if $n$ is even, where $L$ is any closed oriented $n$-manifold.
\end{cor}

In fact, the proof of the above corollary will give us as well a stronger estimate than Theorem \ref{t:main}(a) in almost all cases, namely, $\mathrm{TC}(M\#L)\geq7$ if $n$ is odd and $1<k<n-1$. 

Corollary \ref{c:nc&cs} takes a very strong form in dimension four, since then we conclude that the topological complexity attains its maximum value: 

\begin{cor}\label{c:betti}
Let $M$ be a closed oriented negatively curved $4$-manifold with non-zero second Betti number\footnote{The $i$th-Betti number of a space $X$ is defined by $b_i(X)=\dim H_i(X;\Q)$.}. Then $\mathrm{TC}(M\#L)=9$ for any closed oriented $4$-manifold $L$.
\end{cor}

Of course, similar results with even better bounds can be obtained in higher dimensions as long as there are enough non-zero cohomology classes; we will discuss this and its potential generalisations  in Section \ref{s:final}. 

The cohomology classes $u_a\in H^k(M;\Q)$ of Theorem \ref{t:main} have attracted considerable interest when $N=S^1$, i.e. when $u_a$ are atoirodal, as already mentioned in relation to Mescher's result~\cite[Theorem 6.1]{Me21}. 
Clearly, an atoroidal class $u\in H^2(M;\Q)$ gives rise to a homology class $v\in H_2(M;\Q)$ ($u$'s Kronecker dual) which is {\em not realisable} by tori, that is, there is no continuous map $f\colon T^2\to M$ such that $H_2(f)([T^2])=\alpha\cdot v$, for some $\alpha\neq0$. The (non)-realisation of degree $k$ homology classes by closed oriented connected $k$-manifolds is a classical problem in Topology (known as Steenrod’s problem; see~\cite[Problem 25]{Ei49}), and outstanding contributions date back to the work of Thom~\cite{Th54}.  Our second goal in this paper is to construct new homology classes that are not realisable by connected tori:

\begin{thm}\label{t:torus}
Let $M_1,M_2$ be closed oriented connected manifolds of dimension $n\geq3$. Suppose that there exist non-trivial homology classes $v_j\in H_2(M_j;\Q)$ which are not in the image of the Hurewicz homomorphism. Then
 \[
 v:=(v_1,v_2)\in H_2(M_1\# M_2;\Q)=H_2(M_1;\Q)\oplus H_2(M_2;\Q)
 \]
 is not realisable by a connected torus.
\end{thm}

Note that the $v_j$ in Theorem \ref{t:torus} are called homologically {\em aspherical}; see~\cite{BK08} and also~\cite{RO99},~\cite{GM20}. One of our goals in a previous draft was to use the class $v$ from Theorem \ref{t:torus} in combination with Theorem \ref{t:main} to obtain new bounds for the topological complexity of connected sums. However, as pointed out to me by D. Kotschick and L. Sch\"onlinner this cannot be done, for instance, due to the fact that there is no connectivity requirement for $S^1\times N$ in Theorem \ref{t:main} (cf. Theorem \ref{t:weightsproperties3}) and clearly $v$ in Theorem \ref{t:torus} will be realised by several copies of $T^2$ once each of the $v_j$ is realised by $T^2$; see also Remark \ref{r:nonatoroidal} for an example in the connected setting.

\subsection*{Outline of the paper}
In Section \ref{s:weight} we will recall certain bounds of the sectional category and topological complexity weights. In Section \ref{s:main} we will prove Theorem \ref{t:main}, as well as Corollaries \ref{c:nc&cs} and \ref{c:betti}. In Section \ref{s:torus} we will prove Theorem \ref{t:torus} and explain through a precise example (the connected sum $T^4\# T^4$) that the Kronecker dual of a class $v$ obtained from Theorem \ref{t:torus} need not be atoroidal (Remark \ref{r:nonatoroidal}). Finally, in Section \ref{s:final} we will discuss further applications and open questions motivated by this work. More precisely, we will investigate how the maximum value of the topological complexity 
 is related to the triviality of the center of the fundamental group of aspherical manifolds and to their connected sums.

\subsection*{Notation and terminology}
All (co)homology groups in this paper are taken with {\em rational} coefficients and for brevity we will write $H_*(X)$ instead of $H_*(X;\Q)$. For the same reason, by ``manifold" we will always mean a {\em closed} and {\em oriented} manifold.

\subsection*{Acknowledgments}
I would like to thank Jean-Fran\c cois Lafont and Stephan Mescher for useful discussions, as well as Mark Grant and Dieter Kotschick for suggesting to use aspherical homology classes in Theorem \ref{t:torus} rather than the assumptions of a previous draft. I am particularly thankful to two anonymous referees, whose comments on the proof of Theorem \ref{t:main} yielded a stronger statement in the present version, as well as to Dieter Kotschick and Lukas Sch\"onlinner for pointing out that the classes of  Theorem \ref{t:torus} cannot be applied to obtain new bounds for the topological complexity of connected sums. The hospitality of MPIM Bonn and the University of Halle-Wittenberg are thankfully acknowledged.

\section{Weight bounds}\label{s:weight}

In this section we give some preliminaries and record some results that we will need about sectional category and topological complexity weights.

\subsection{LS-category and topological complexity as Schwarz genus} 
The {\em Schwarz genus} or {\em sectional category} of a (Hurewicz) fibration $p\colon E\to B$, denoted by $\mathfrak{genus}(p)$ is the smallest integer $n$ for which the base $B$ can be covered by $n$ open sets $U_1,...,U_n$, such that for each $i=1,...,n$, there is a continuous map $s_i: U_i\to E$ satisfying $p\circ s_i = id_{U_i}$.

The Schwarz genus can be used to define  both LS-category and topological complexity: Let $X$ be a topological space with base point $x_0\in X$. Consider the Serre path fibration 
\[
p_X\colon P_0X=\{\gamma\colon[0,1]\to X \ | \ \gamma(0)=x_0\}\to X, \ \gamma\mapsto\gamma(1).
\]
The Schwarz genus of $p_X$ is the {\em LS-category} of $X$,
\[
\mathrm{cat}(X) = \mathfrak{genus}(p_X).
\]
Furthermore, we consider the free path fibration
\[
\pi_X\colon X^{[0,1]}\to X\times X, \ \gamma\mapsto(\gamma(0),\gamma(1)),
\]
where $X^{[0,1]}$ is equipped with the compact-open topology. The Schwarz genus of $\pi_X$ is the {\em topological complexity} of $X$,
\[
\mathrm{TC}(X)= \mathfrak{genus}(\pi_X).
\]

\subsection{Sectional category and topological complexity weights}

Given a (Hurewicz) fibration $p\colon E\to B$ and a continuous map $f\colon Y\to B$, let $f^*p\colon f^*E\to Y$ denote the pullback fibration.

\begin{defn}\label{d:sewgt}
 The {\em sectional category weight with respect to} $p$ of a non-trivial cohomology class $u\in H^*(B)$ is defined by
\[
\mathrm{wgt}_p(u)=\sup\{k \ | \ f^*(u)=0 \ \text{for all maps} \ f\colon Y\to B \ \text{with} \ \mathfrak{genus}(f^*p)\leq k\},
\]
where $Y$ is any topological space.
\end{defn}

We will need the following results:

\begin{thm}\cite{Ru99,FG07}\label{t:weightsproperties1}
If $p\colon E\to B$ is a fibration and $u\in H^*(B)$ a non-zero cohomology class, then the following hold:
\begin{itemize}
\item[(a)] $\mathfrak{genus}(p)\geq\mathrm{wgt}_p(u) + 1$;
\item[(b)] $\mathrm{wgt}_p(u)\leq |u|$, where $|u|$ denotes the degree of $u$.
\end{itemize}
\end{thm}

\begin{thm}\cite[Prop. 32]{FG07}\label{t:weightsproperties2}
Let $p\colon E\to B$ be a fibration. Suppose that there exist cohomology classes $u_1,...,u_k\in H^*(B)$ with non-trivial cup product $u_1\cup\cdots\cup u_k\neq0$. Then
\[
\mathrm{wgt}_p(u_1\cup\cdots\cup u_k)\geq\sum_{i=1}^k\mathrm{wgt}_p(u_i).
\]
\end{thm}

\begin{defn}\label{d:tcwgt}
Given a space $X$, the {\em topological complexity weight} of a non-trivial cohomology class $\overline u\in H^*(X\times X)$ is defined by
\[
\mathrm{wgt}_{\mathrm{TC}}(\overline u)=\mathrm{wgt}_{\pi_X}(\overline u),
\]
where $\pi_X$ is the free path fibration.
\end{defn}

Recall that the (rational) cohomology ring $H^*(X)$ is a graded $\Q$-algebra with the cup product multiplication
\[
H^*(X)\otimes H^*(X)\to H^*(X),
\]
and the kernel of this map is the ideal of the zero-divisors of $H^*(X)$. Using that $\mathrm{TC}(X)= \mathfrak{genus}(\pi_X)$, Farber showed that the topological complexity of $X$ is greater than the zero-divisors-cup-length of $H^*(X)$~\cite[Theorem 7]{Fa03}. An easy example which illustrates this situation is given when $X=S^n$: Suppose $\omega_{S^n}\in H^n(S^n)$ is the cohomological fundamental class of $S^n$ and consider the zero divisor
\[
\overline{\omega_{S^n}}=1\times\omega_{S^n}-\omega_{S^n}\times1\in H^n(S^n\times S^n).
\]
Then
\[
\overline{\omega_{S^n}}^2=((-1)^{n+1}-1)\omega_{S^n}\times\omega_{S^n},
\]
which means that $\mathrm{TC}(S^n)\geq2$ if $n$ is odd and $\mathrm{TC}(S^n)\geq3$ if $n$ is even. In fact, both inequalities are equalities~\cite[Theorem 8]{Fa03}.

Because of Theorem \ref{t:weightsproperties1}(a) and Definition \ref{d:tcwgt}, it is evident that bounding below $\mathrm{wgt}_{\pi_X}$ will give us a lower bound for $\mathrm{TC}(X)$. The following result of Mescher is towards this direction:

\begin{thm}\cite[Prop. 5.3]{Me21}\label{t:weightsproperties3}
Let $X$ be a topological space and let $u\in H^k(X)$ be a non-zero cohomology class with $k\geq2$. If $H^k(f)(u) = 0$ for all continuous maps $f\colon N\times S^1\to X$, where $N$ is any $(k -1)$-manifold, then 
\[
\mathrm{wgt}_{\mathrm{TC}}(\overline u)\geq 2,
\]
where $\overline u =1\times u-u\times1\in H^k(X\times X)$.
\end{thm}

\section{TC bounds via cohomology classes}\label{s:main}

In this section, we prove Theorem \ref{t:main} and Corollaries \ref{c:nc&cs} and \ref{c:betti}.

\subsection{Proof of Theorem \ref{t:main}}
Let $u_a\in H^k(M)$, $1<k<n$,  be a non-zero cohomology class such that $H^k(f)(u_a) = 0$ for all continuous maps $f\colon N\times S^1\to M$ and any $(k -1)$-manifold $N$. Then, by Theorem \ref{t:weightsproperties3}, the class $\overline{u_a} =1\times u_a-u_a\times 1\in H^k(M\times M)$ satisfies $\mathrm{wgt}_{\mathrm{TC}}(\overline{u_a})\geq 2$. We have
\begin{equation}\label{eq:u_a^2}
\begin{aligned}
\overline{u_a}^2&=(1\times u_a-u_a\times 1)\cup(1\times u_a-u_a\times 1)\\
&=1\times{u_a}^2-(-1)^ku_a\times u_a-u_a\times u_a+{u_a}^2\times1\in H^{2k}(M\times M).
\end{aligned}
\end{equation}

By Poincar\'e Duality, there exists a non-zero cohomology class $u_b\in H^{n-k}(M)$, where $n-k\geq1$, such that $u_a\cup u_b=\omega_M\in H^n(M)$. Let $\overline{u_b}=1\times u_b-u_b\times 1\in H^{n-k}(M\times M)$.  

Since $k$ is even, equation \eqref{eq:u_a^2} becomes
\begin{equation}\label{eq:u_a^2even}
\overline{u_a}^2=1\times{u_a}^2-2u_a\times u_a+{u_a}^2\times1,
\end{equation}
which is not zero because of the term $u_a\times u_a$. 
Using \eqref{eq:u_a^2even}, we compute
\begin{equation}\label{eq:firstcup}
\begin{aligned}
\overline{u_a}^2\cup\overline{u_b}
&=(1\times{u_a}^2-2u_a\times u_a+{u_a}^2\times1)(1\times{u_b}-{u_b}\times 1)\\
&=-u_b\times{u_a}^2-2u_a\times\omega_M+2\omega_M\times u_a+{u_a}^2\times{u_b}\in H^{n+k}(M\times M).
\end{aligned}
\end{equation}
We observe that $\overline{u_a}^2\cup\overline{u_b}\neq 0$ because $-u_a\times\omega_M+\omega_M\times u_a\neq0$. (Note, also, that the first two terms $-u_b\times{u_a}^2,-2u_a\times\omega_M$ in \eqref{eq:firstcup} cannot cancel each other because in that case we would have $u_a^2=\mu_1\omega_M$ and $u_b=\mu_2u_a$ for $\mu_1\mu_2=-2$, which would contradict that $\omega_M=u_a\cup u_b$; similarly, the last two terms of \eqref{eq:firstcup} cannot cancel each other. In fact, as we shall see below, the case $u_a^2=\mu_1\omega_M$, $\mu_1\neq0$, will imply a much better bound for $\mathrm{TC}$.) Thus, Theorem \ref{t:weightsproperties1}(a), Theorem \ref{t:weightsproperties2} and the fact that $\mathrm{wgt}_{\mathrm{TC}}(\overline{u_a})\geq 2$ 
imply
\begin{equation}\label{TCfirst}
\begin{aligned}
\mathrm{TC}(M)&=\mathfrak{genus}(\pi_X)
\geq\mathrm{wgt}_{\mathrm{TC}}(\overline{u_a}^2\cup\overline{u_b})+1
\geq2\mathrm{wgt}_{\mathrm{TC}}(\overline{u_a})+\mathrm{wgt}_{\mathrm{TC}}(\overline{u_b})+1\geq6.
\end{aligned}
\end{equation}
This bound proves Theorem \ref{t:main}(a) and extends~\cite[Theorem 6.1]{Me21}, where $N=S^1$.

Furthermore, we compute 
\[
\begin{aligned}
\overline{u_a}^2\cup\overline{u_b}^2&=(-u_b\times{u_a}^2-2u_a\times\omega_M+2\omega_M\times u_a+{u_a}^2\times u_b)(1\times u_b-u_b\times 1)\\
&={u_b}^2\times{u_a}^2+(-1)^n2\omega_M\times\omega_M+2\omega_M\times\omega_M+{u_a}^2\times{u_b}^2\in H^{2n}(M\times M).
\end{aligned}
\]
If $n$ is even and $n\neq2k$, then $\overline{u_a}^2\cup\overline{u_b}^2\neq0$, because then $(-1)^n\omega_M\times\omega_M+\omega_M\times\omega_M\neq0$ and ${u_b}^2\times{u_a}^2={u_a}^2\times{u_b}^2=0$ since  one of $k$ or $n-k$ is greater than the middle dimension $n/2$. If $n=2k$ and $u_a^2=0$, then for the same reason $\overline{u_a}^2\cup\overline{u_b}^2\neq0$. Hence, Theorem \ref{t:weightsproperties1}(a), Theorem \ref{t:weightsproperties2} and the fact that $\mathrm{wgt}_{\mathrm{TC}}(\overline{u_a})\geq 2$  imply 
\begin{equation}\label{TCsecond}
\begin{aligned}
\mathrm{TC}(M)&=\mathfrak{genus}(\pi_X)
\geq\mathrm{wgt}_{\mathrm{TC}}(\overline{u_a}^2\cup\overline{u_b}^2)+1
\geq2\mathrm{wgt}_{\mathrm{TC}}(\overline{u_a})+2\mathrm{wgt}_{\mathrm{TC}}(\overline{u_b})+1\geq7.
\end{aligned}
\end{equation}

Finally, if $n=2k$ and  $u_a^2\neq0$, then $u_a^2=\mu\omega_M$ for some $\mu\neq0$. Thus \eqref{eq:u_a^2even} becomes
\begin{equation*}\label{eq:u_a^22k}
\begin{aligned}
\overline{u_a}^2=1\times\mu\omega_M-2u_a\times u_a+\mu\omega_M\times1,
\end{aligned}
\end{equation*}
and we compute
\[
\begin{aligned}
\overline{u_a}^2\cup\overline{u_a}^2&=(1\times\mu\omega_M-2u_a\times u_a+\mu\omega_M\times1)^2\\
&=\mu^2\omega_M\times\omega_M+4\mu^2\omega_M\times\omega_M+\mu^2\omega_M\times\omega_M\\
&=6\mu^2\omega_M\times\omega_M\neq0.
\end{aligned}
\]
Thus, again by Theorem \ref{t:weightsproperties1}(a), Theorem \ref{t:weightsproperties2} and the fact that $\mathrm{wgt}_{\mathrm{TC}}(\overline{u_a})\geq 2$  we obtain 
\begin{equation}\label{TCthird}
\begin{aligned}
\mathrm{TC}(M)&=\mathfrak{genus}(\pi_X)
\geq\mathrm{wgt}_{\mathrm{TC}}(\overline{u_a}^2\cup\overline{u_a}^2)+1
\geq2\mathrm{wgt}_{\mathrm{TC}}(\overline{u_a})+2\mathrm{wgt}_{\mathrm{TC}}(\overline{u_a})+1\geq9.
\end{aligned}
\end{equation}

This proves Theorem \ref{t:main}(b) and (c), and completes the proof of Theorem \ref{t:main}.

\begin{rem}
Note that the above proof does not work when $k$ has odd degree. In that case, $u_a^2=0$ and therefore \eqref{eq:u_a^2} becomes $\overline{u_a}^2=0$.
\end{rem}

\subsection{Proof of Corollary \ref{c:nc&cs}}

First, we recall the definitions and results about some of the basic ingredients for the proof of Corollary \ref{c:nc&cs}.

Given a space $X$ and a rational homology class $v\in H_n(X)$, Gromov~\cite{Gr82} defines 
the {\em $\ell^1$-semi-norm} (also known as the {\em Gromov norm}) of $v$ by
\[
\|v\|_1:=\inf_c\biggl\{\sum_{j} |\lambda_j| \ \biggl\vert  \ \sum_j \lambda_j\sigma_j\in C_n(X) \ \text{is a rational cycle representing } v  \biggl\}.
\]
When $M$ is an $n$-manifold, then $\|M\|:=\|[M]\|_1$ is called the {\em simplicial volume} of $M$. As usual, $[M]$ denotes the fundamental class of $M$. Amongst the most important properties of the $\ell^1$-semi-norm is the functorial property, which follows immediately from the definition:

\begin{lem}\label{l:functorial}
If $f\colon X\to Y$ is a continuous map and $v\in H_n(X)$, then $\|v\|_1\geq\|H_n(f)(v)\|_1$.
\end{lem}

Inoue-Yano~\cite{IY81} and Gromov~\cite{Gr82} proved the following notable result about negatively curved manifolds:

\begin{thm}\cite{IY81,Gr82}\label{t:IYG}
Let $M$ be a negatively curved manifold. If $k>1$, then $\|v\|_1>0$ for each $v\in H_k(X)$.
\end{thm}

Let now $u\in H^k(M)$, $k>1$, where $M$ is a negatively curved manifold. Suppose that there is a $(k-1)$-manifold $N$ and a continuous map $f\colon S^1\times N\to M$ such that $H^k(f)(u)=\beta\cdot\omega_{S^1\times N}$ for some $\beta\neq0$. Then, using the Kronecker product $\langle\cdot,\cdot\rangle$, we obtain
 \begin{equation}\label{eq:nc}
\begin{aligned}
\beta&=\langle\beta\cdot\omega_{S^1\times N}, [S^1\times N]\rangle\\
& =\langle H^k(f)(u), [S^1\times N]\rangle\\
& =\langle u, H_k(f)([S^1\times N])\rangle\\
& =\langle u, \beta\cdot(v+w)\rangle,
\end{aligned}
\end{equation}
i.e., 
\begin{equation}\label{eq:nchomology}
H_k(f)([S^1\times N]) = \beta\cdot(v+w),
\end{equation}
where $v,w\in H_k(M)$ are such that $\langle u, v\rangle=1$ and $\langle u, w\rangle=0$.

Since $S^1\times N$ admits self-maps of degree greater than one (because $S^1$ does), we have that $\|S^1\times N\|=0$ by Lemma \ref{l:functorial}. Hence, applying again Lemma \ref{l:functorial} and \eqref{eq:nchomology} we obtain
\[
0=\|S^1\times N\|\geq\|H_k(f)([S^1\times N])\|_1=\|\beta\cdot(v+w)\|_1.
\]
But this is a contradiction, since $k>1$, and thus $\|\beta\cdot(v+w)\|_1>0$ by Theorem \ref{t:IYG}. Together with Mescher's Theorem \ref{t:weightsproperties3}, we obtain the following:

\begin{thm}\label{t:nc}
Let $M$ be a negatively curved $n$-manifold and $u\in H^k(M)$ be any non-zero cohomology class, where $k\geq2$. Then
\[
\mathrm{wgt}_{\mathrm{TC}}(\overline u)\geq 2,
\]
where $\overline u =1\times u-u\times1\in H^k(M\times M)$.
\end{thm}

Now Corollary \ref{c:nc&cs} for $M$ follows form Theorem \ref{t:main}(b), where the improvement $\mathrm{TC}\geq9$ for $n$ even follows form the fact that both classes $\overline{u_a}$ and $\overline{u_b}$ that appear in computation \eqref{TCsecond} have $\mathrm{wgt}_{\mathrm{TC}}\geq2$ by Theorem \ref{t:nc}.

Finally, the same conclusion follows for any connected sum $M\#L$ using the cohomology classes of $M$. (Recall also that cup products of cohomology classes coming from different summands in a connected sum vanish.)

\begin{rem}\label{r:TC7}
Note that the above argument implies as well that $\mathrm{TC}(M)\geq7$ when $n$ is odd and $n-k\neq1$. (Similarly for $M\#L$.)
\end{rem}

\subsection{Proof of Corollary \ref{c:betti}}
The bound $\mathrm{TC}(M\#L)\geq9$ follows by Corollary \ref{c:nc&cs}. The bound $\mathrm{TC}(M\#L)\leq9$ holds for any $4$-manifold by \eqref{eq:dim-TC}.

\section{Homology classes in connected sums}\label{s:torus}

In this section, we prove Theorem \ref{t:torus}.

\medskip

Recall that, given a topological space $X$, there is an exact sequence
\begin{equation}\label{Hopf}
 \pi_2(X)\xrightarrow{h} H_2(X) \xrightarrow{H_2(c_{X})} H_2(B\pi_1(X))\to0,
 \end{equation}
 called the {\em Hopf sequence}, where $h$ denotes the Hurewicz homomorphism and $c_X$ (resp. $B(\pi_1(X))$) the classifying map (resp. space); see~\cite[Ch.II, Sec. 5]{Br82}. 
 
\begin{defn}
A (rational) homology class $v\in H_2(X)$ which is not in the image of the (rational) Hurewicz homomorphism is called {\em aspherical}. 
\end{defn} 
 
We have the following key observation:

\begin{lem}\label{l:injective}
Let $X$ be a topological space which admits a universal covering. If an aspherical  class $v\in H_2(X)$ is realised by a torus, then it is realised by a $\pi_1$-injective map $T^2\to X$.
\end{lem}
\begin{proof}
Let $v\in H_2(X)$ be an aspherical class and suppose that $f\colon T^2\to X$ is such that $H_2(f)([T^2])=\alpha\cdot v$, for some $\alpha\in\Q\setminus\{0\}$. Let $f_*\colon\pi_1(T^2)\to\pi_1(X)$ be the induced homomorphism. Since $\pi_1(T^2)\cong\Z^2$, we can pass to a finite covering $q\colon T^2\to T^2$ so that either $(f\circ q)_*(\pi_1(T^2))$ is trivial or infinite cyclic or $\Z^2$ (these will be already the only cases, without passing to further coverings, if $\pi_1(X)$ is torsion-free). Let us denote by $g$ the composite map
\[
g:=f\circ q\colon T^2\to X.
\]
We have that
\[
H_2(g)([T^2])=H_2(f\circ q)([T^2])=\deg(q)\alpha\cdot v,
\]
and so $g$ realises $v$. 

Suppose that $g_*\colon\pi_1(T^2)\to\pi_1(X)$  is not injective. Then, either $g_*(\pi_1(T^2))$ is trivial or infinite cyclic. Denote by $p\colon\overline X\to X$ the covering corresponding to $g_*(\pi_1(T^2))$.

If $g_*(\pi_1(T^2))$ is trivial, then $B\pi_1(\overline X)\simeq\mathrm{pt}$, and if $g_*(\pi_1(T^2))$ is infinite cyclic, then $B\pi_1(\overline X)\simeq S^1$. In either case, we have that $H_2(B\pi_1(\overline X))=0$ and thus there is a surjection
\[
 \pi_2(\overline X)\xrightarrow{\overline h} H_2(\overline X)
\]
obtained by the Hopf sequence \eqref{Hopf} for $\overline X$. But this means that in either case $v$ is in the image of $h\colon\pi_2(X)\to H_2(X)$, i.e. it is spherical. This contradiction completes the proof.
\end{proof}

We are now ready to prove Theorem \ref{t:torus}:

\begin{proof}[Proof of Theorem \ref{t:torus}]
Let the homology class
 \[
 v:=(v_1,v_2)\in H_2(M_1\# M_2)=H_2(M_1)\oplus H_2(M_2),
 \]
such that each $v_j\in H_2(M_j)$ is a non-trivial aspherical class, that is, $v_j\notin\mathrm{im}(h_{M_j})$, where $h_{M_j}\colon\pi_2(M_j)\to H_2(M_j)$ is the Hurewicz homomorphism for $M_j$.

First, we observe that $v$ is an aspherical class. Indeed, suppose that $v$ is in the image of the Hurewicz homomorphism
\[
h_{M_1\#M_2}\colon\pi_2(M_1\#M_2)\to H_2(M_1\#M_2),
\]
that is, there is some $g\colon S^2\to M_1\#M_2$, such that
\[
v=h_{M_1\#M_2}([g])=H_2(g)([S^2]), \ [g]\in\pi_2(M_1\# M_2).
\]
For $j=1,2$, let $p_j\circ g\colon S^2\to M_j$, where $p_j\colon M_1\# M_2\to M_j$ denotes the degree one pinch map. Then $[p_j\circ g]\in\pi_2(M_j)$ and
\[
h_{M_j}([p_j\circ g])=H_2(p_j\circ g)([S^2])=H_2(p_j)(v)=v_j,
\]
which means that the $v_j$ is in the image of the Hurewicz homomorphism, contradicting our assumption.

Suppose now that there is a continuous map 
$f\colon T^2\to M_1\# M_2$ 
such that 
\begin{equation}\label{eq:atoroidal}
H_2(f)([T^2])=\alpha\cdot v=\alpha\cdot (v_1,v_2), \ \text{for some} \ \alpha\neq0.
\end{equation}

Let $\pi_1(T^2)=\pi_1(S^1_1\times S^1_2)=\Z\times\Z$, where we can omit the base-point, since $T^2$ is assumed to be connected.  Since $v$ is aspherical, 
we can assume  by Lemma \ref{l:injective} that the induced homomorphism
\begin{equation*}\label{eq:pi_1}
f_*\colon \pi_1(T^2)\to\pi_1(M_1)\ast\pi_1(M_2)
\end{equation*}
is injective. But $f_*(\pi_1(T^2))\subseteq\pi_1(M_1)\ast\pi_1(M_2)$ is itself a free product of conjugates of subgroups of the $\pi_1(M_j)$ and of a free group, by the classic Kurosh subgroup theorem~\cite{Ku34}. Hence, $\pi_1(S^1_1\times S^1_2)$ lies in exactly one of the $\pi_1(M_j)$. 

As above, for $j=1,2$, let $p_j\circ f\colon T^2\to M_j$, where $p_j\colon M_1\# M_2\to M_j$ denotes the degree one pinch map. Set 
\begin{equation}\label{eq:Zcopies}
H_{j,1}=\mathrm{im}((p_j\circ f|_{S^1_1})_*) \ \text{and}  \ H_{j,2}=\mathrm{im}((p_j\circ f|_{S^1_2})_*), \ j=1,2.
\end{equation}
Then the groups $H_{j,1}, H_{j,2}\subseteq\pi_1(M_j)$ 
commute elementwise, and thus there is a well-defined homomorphism given by the multiplication map
\begin{equation}\label{multiplication}
\varphi_j\colon H_{j,1}\times H_{j,2}\to\pi_1(M_j),
\end{equation}
for each $j=1,2$. Let the maps  induced on the classifying spaces
 \[
 h_{j,i}=B(p_j\circ f|_{S^1_i})_*\colon B\pi_1(S^1_i)\to BH_{j,i}, \ i=1,2,
 \]
 \[
 B\varphi_j\colon BH_{j,1}\times BH_{j,2}\to B\pi_1(M_j).
 \]
Note that $B\pi_1(S^1_i)\simeq S^1_i$ and the classifying maps $c_{S^1_i}\colon S^1_i\to B\pi_1(S^1_i)$ are homotopic to the identity for $i=1,2$. Hence, we obtain, for each $j=1,2$, the following homotopy commutative diagram (cf.~\cite[Prop 2.2]{KL09}) 
\begin{equation}\label{commutative2}
\xymatrix{
S^1_1\times S^1_2\ar[d]_{g_j} \ar[r]^{p_j\circ f}&  \ar[d]^{c_{M_j}}  M_j\\
BH_{j,1}\times BH_{j,2} \ar[r]^{\ \ \ B\varphi_j}& B\pi_1(M_j) \\
}
\end{equation}
where
\[
\begin{aligned}
g_j=( h_{j,1}\circ c_{S^1_1})\times(h_{j,2}\circ c_{S^1_2})
\simeq  h_{j,1}\times h_{j,2},
\end{aligned}
\]
and $c_{M_j}$ denotes the classifying map. Since $v_j$ are aspherical, the Hopf sequence \eqref{Hopf} tells us that $H_2(c_{M_j})(v_j)\neq0$ for each $j=1,2$, and so we obtain by \eqref{eq:atoroidal}
\begin{equation}\label{eq:commuteconnected}
\begin{aligned}
H_2(c_{M_j}\circ p_j\circ f)([S^1_1\times S^1_2])&=\alpha\cdot H_2(c_{M_j}\circ p_j)((v_1,v_2))\\
&=\alpha\cdot H_2(c_{M_j})(v_j)\neq 0\in H_2(M_j).
\end{aligned}
\end{equation}
Hence, the commutative diagram \eqref{commutative2} tells us that
\[
0\neq H_2(B\varphi_j\circ g_j)([S^1_1\times S^1_2])=H_2(B\varphi_j)(H_1(h_{j,1})([S^1_1])\times H_1(h_{j,2})([S^1_2])).
\]
In particular, the rational homology groups $H_1(BH_{j,i})$ are not trivial for both $i$, and so both groups $H_{j,1},H_{j,2}$ are infinite. Therefore, $H_{j,i}\cong\Z$, $i=1,2$, and \eqref{multiplication} takes the form
\begin{equation}\label{eq:bothinjective}
\varphi_j\colon\Z\times\Z\to\pi_1(M_j),
\end{equation}
for both $j$. But this implies that $\pi_1(S^1_1\times S^1_2)$ lies in both $\pi_1(M_j)$, which is a contradiction. Hence the map $f$ cannot exist, and the homology class $v$ is cannot be represented by a connected $T^2$. 
\end{proof}

\begin{rem}\label{r:nonatoroidal}
Note that Kronecker duals of $v$ of Theorem \ref{t:torus} need not be atoroidal: Let $M_1=M_2=T^4$ and consider the class
\[
v=[T^2]\times 1 + 1\times[T^2]\in H_2(T^4)\oplus H_2(T^4)=H_2(T^4\# T^4).
\]
The classes $[T^2]\times 1$, $1\times[T^2]$ are clearly aspherical and thus $v$ is not realisable by a connected $T^2$ by Theorem \ref{t:torus}. Consider a Kronecker dual of $v$ given by
\[
u=\omega_{T^2}\times1+1\times\omega_{T^2}\in H^2(T^4)\oplus H^2(T^4)=H^2(T^4\# T^4).
\]
Let $f\colon T^2\to T^4\# T^4$ be the inclusion of a copy of $T^2=T^2\times\mathrm{pt}$ in the first summand $T^4=T^2\times T^2$. Then we compute
\[
\begin{aligned}
\langle H^2(f)(u),[T^2]\rangle&=\langle u,H_2(f)([T^2])\rangle\\
& =\langle\omega_{T^2}\times1+1\times\omega_{T^2},[T^2]\times 1\rangle\\
& =\langle\omega_{T^2}\times1,[T^2]\times 1\rangle\\
& =1.
\end{aligned}
\]
Therefore $H^2(f)(u)\neq0$. In particular, Theorem \ref{t:weightsproperties3} (or Theorem \ref{t:main}) does not apply to $u$.
\end{rem}

\section{Further applications and open problems}\label{s:final}

We will end our discussion with some final applications and open questions motivated by the results of this paper.

\subsection{Positivity of the $\ell^1$-semi-norm}

Given a group $G$, we denote its topological complexity by $\mathrm{TC}(G)=\mathrm{TC}(BG)$. Thus, if $M$ is an aspherical manifold, i.e. $\pi_k(M)=0$ for all $k\geq2$, then one has $\mathrm{TC}(\pi_1(M))=\mathrm{TC}(M)$. Negatively curved manifolds are aspherical (by Cartan-Hadamard theorem), thus their topological complexity can be identified with that of their fundamental group.

Dranishnikov~\cite{Dr20} proved that torsion-free hyperbolic groups $G$ attain maximal topological complexity, equal to $2\mathrm{cd}(G)+1$. Thus, in particular, every negatively curved $n$-manifold $M$ has topological complexity $2n+1$.  Corollaries \ref{c:nc&cs} and \ref{c:betti} give an alternative proof of Dranishnikov's result in certain cases under cohomological conditions (instead of group theoretic), using the rich structure of these spaces via the $\ell^1$-semi-norm. In high dimensions, the proof of Corollaries \ref{c:nc&cs} and \ref{c:betti} can clearly give much better bounds than $\mathrm{TC}\geq9$ as long as there are enough cohomology classes whose cup product gives us the fundamental class of our manifold. Indeed, in such situations we can apply the estimate of Theorem \ref{t:weightsproperties2} together with Theorem \ref{t:nc}.

Therefore we have the following natural questions regarding the role of the $\ell^1$-semi-norm and trying to extend Theorem \ref{t:nc} and Dranishnikov's result:

\begin{que}\label{que1}
Let $M$ be an aspherical $n$-manifold.
\begin{itemize}
\item[(a)] For $k>1$, suppose that $H_k(M)$ contains a homology class with positive $\ell^1$-semi-norm. Is there a cohomology class $u\in H^k(M)$ with $\mathrm{wgt}_{\mathrm{TC}}(\overline u)\geq k$?
\item[(b)] Suppose $\|M\|>0$, i.e. $M$ has positive simplicial volume. Does it hold $\mathrm{TC}(M)=2n+1$? 
\end{itemize}
\end{que}

Note that a positive answer to Question \ref{que1}(a) means in fact that $\mathrm{wgt}_{\mathrm{TC}}(\overline u)= k$ by Theorem \ref{t:weightsproperties1}(b). Also, when $n$ is even, then an affirmative answer to Question \ref{que1}(a) implies an affirmative answer to Question \ref{que1}(b). Indeed, if $\|M\|>0$, then $H^n(f)(\omega_M)=0$ for any $f\colon S^1\times N\to M$ and suppose $\mathrm{wgt}_{\mathrm{TC}}(\overline{\omega_M})= n$. Then, since $\overline{\omega_M}^2=-2\omega_M\times\omega_M\neq0$, Theorem \ref{t:weightsproperties2} would imply $\mathrm{TC}(M)=2n+1$.

\subsection{Connected sums and asphericity}

Pinching to a point the essential sphere of a connected sum $M_1\# M_2$, we obtain the wedge sum $M_1\vee M_2$, whose topological complexity admits the bounds
\begin{equation}\label{wedge}
\max\{\mathrm{TC}(M_1), \mathrm{TC}(M_2), \mathrm{cat}(M_1\times M_2)\}\leq\mathrm{TC}(M_1\vee M_2)\leq\mathrm{TC}(M_1) +\mathrm{TC}(M_2)-1,
\end{equation}
whenever $\mathrm{TC}(M_i)>\dim(M_i)+1$, for both $i$; see~\cite[Theorems 2.5 and 3.6]{Dr14}. A natural problem is to compare the topological complexities of $M_1\# M_2$ and $M_1\vee M_2$, and, more precisely, whether the inequality 
\begin{equation}\label{eq:wedge-cs}
\mathrm{TC}(M_1\vee M_2)\geq\mathrm{TC}(M_1\# M_2)
\end{equation}
holds. Dranishnikov and Sadykov showed \eqref{eq:wedge-cs} 
for $r$-connected $n$-manifolds $M_i$ such that $\mathrm{TC}(M_i)\geq\frac{n+2}{r+1}+1$ for at least one $i$~\cite[Theorem 1.3]{DS19}. Hence, another interesting fact about Corollary \ref{c:betti}  
is that it shows that not only \eqref{eq:wedge-cs} holds, but it is in fact an equality for the manifolds of Corollary \ref{c:betti} in the aspherical setting:

\begin{cor}
 Let $M_1$ be a negatively curved $4$-manifold with non-zero second Betti number. Then $\mathrm{TC}(M_1\vee M_2)=\mathrm{TC}(M_1\# M_2)$ for any aspherical manifold $M_2$.
\end{cor}
\begin{proof}
We have
 \[
\begin{aligned}
\mathrm{TC}(M_1\vee M_2)&=\mathrm{TC}(B\pi_1(M_1\# M_2)) \  \text{(because} \ M_i\ \text{are aspherical)}\\
& =\mathrm{TC}(\pi_1(M_1)\ast\pi_1(M_2))\\
& =\max\{\mathrm{TC}(\pi_1(M_i)),\mathrm{cd}(\pi_1(M_1)\times\pi_1(M_2))+1\} \ \text{(by \cite[Theorem 2]{DS19b})}\\
& =9 \  \text{(because} \ \mathrm{cd}(\pi_1(M_i))=4 \ \text{by asphericity, or} \ \mathrm{TC}(\pi_1(M_1))=9 \ \text{by \cite{Dr20})}\\
&=\mathrm{TC}(M_1\#M_2) \ \text{(by Corollary \ref{c:betti}).}
\end{aligned}
\]
\end{proof}

In general, very little is known about the topological complexity of connected sums, and how it is related to the topological complexity of the summands of the connected sum. On the one hand, connected sums of the projective space $\R P^n$ have topological complexity $\mathrm{TC}(\#_m\R P^n)=2n+1$ if $m\geq2$ (see~\cite[Theorem 1.3]{CV19}), while $\mathrm{TC}(\R P^n)$ varies, depending on the dimension; more precisely, in all but dimensions $1,3$ and $7$, the topological complexity of $\R P^n$ is given by the smallest integer $k$ such that $\R P^n$ admits an immersion into $\R^{k-1}$ (\cite{FTY03}), which is smaller than $2n+1$. Note that $\R P^n$ is orientable if and only if $n$ is odd. Moreover, $\R P^n$ is $2$-fold covered by $S^n$, thus, in particular, it is a rational homology sphere. Of course, any connected sum $\#_m\R P^n$ is also a rational homology sphere. Clearly, this example is different from the concepts of this paper, since here we need non-trivial intermediate rational homology groups. On the other hand, the topological complexity of non-trivial connected sums of $n$-manifolds does not necessarily attain its maximum value $2n+1$ according to the following examples:

\begin{ex}\label{ex:4dimnonaspherical}
Let $\Sigma_g$ be an aspherical surface of genus $g$. Then by~\cite[Theorem, p. 2511]{Ru97}
\[
\mathrm{cat}(\Sigma_g\times S^2)=\mathrm{cat}(\Sigma_g)+1=4.
\]
Thus, by \eqref{eq:LS-TC} and \cite[Theorem 1.1]{DS19}, we obtain
\begin{equation}\label{surface-sphereupper}
\mathrm{TC}((\Sigma_g\times S^2)\#(\Sigma_g\times S^2))\leq2\mathrm{cat}(\Sigma_g\times S^2)-1=7<9.
\end{equation}
\end{ex}

\begin{ex}\label{ex:productsspheres}
For any spheres $S^m, S^n$ of dimensions $m$ and $n$ respectively we have
\[
\mathrm{TC}((S^m\times S^n)\#(S^m\times S^n))\leq2\mathrm{cat}(S^m\times S^n)-1=5<6.
\]
\end{ex}

The summands $S^m\times S^n$ in Example \ref{ex:productsspheres} are not rational homology spheres, but still they do not contain any class $u_a$ as in Theorem \ref{t:main} (in contrast to Example \ref{ex:4dimnonaspherical}). Furthermore, in the case $m=1,n>1$, the manifold $(S^1\times S^n)\#(S^1\times S^n)$ has non-trivial fundamental group, isomorphic to a free group of two generators $F_2$ (in particular, hyperbolic with trivial center), which has topological complexity equal to $3$ by~\cite{Dr20}. The  last two examples concern manifolds that are not aspherical, but their fundamental groups have trivial center. It is therefore natural to ask about the relationship between triviality of the center of the fundamental group and maximal topological complexity for aspherical manifolds and their connected sums:

\begin{que}\label{que2}
Let $M$ be an aspherical $n$-manifold. 
\begin{itemize}
\item[(a)] If the center $Z(\pi_1(M))$ is trivial, does it hold that $\mathrm{TC}(M)=2n+1$?
\item[(b)] Does it hold $\mathrm{TC}(M\#L)=2n+1$ for any aspherical $n$-manifold $L$?
\end{itemize}
\end{que}

The condition ``trivial center" is automatically satisfied in part (b) of Question \ref{que2}. Our motivation for part (a) of  Question \ref{que2} partially stems from the following examples:

\begin{ex}\label{ex:n-torus}
The $n$-torus $T^n$ is aspherical with fundamental group $\pi_1(T^n)=\Z^n=Z(\pi_1(T^n))$. It has topological complexity $\mathrm{TC}(T^n)=n+1$; see~\cite[Theorem 13]{Fa03}.
\end{ex}

\begin{ex}\label{ex:product-surface}
Let $F$ be a surface and $\Sigma_g$ be a hyperbolic surface. If $F=S^2$ or $T^2$, then the product $F\times\Sigma_g$ satisfies
\[
\mathrm{TC}(F\times\Sigma_g)\geq7,
\]
by Theorem \ref{t:main}(b). Moreover, since $\mathrm{TC}(\Sigma_g)=5$, by~\cite[Theorem 9]{Fa03} and $\mathrm{TC}(T^2)=\mathrm{TC}(S^2)=3$, by~\cite[Theorem 13]{Fa03}, we obtain 
\[
\mathrm{TC}(F\times\Sigma_g)\leq\mathrm{TC}(F)+\mathrm{TC}(\Sigma_g)-1\leq3+5-1=7,
\]
by~\cite[Theorem 11]{Fa03}. Thus,
\[
\mathrm{TC}(F\times\Sigma_g)=7,  \ \text{if} \ F=S^2 \ \text{or} \ T^2.
\]
Finally, if $F$ is hyperbolic, then the same proof as in Corollary \ref{c:nc&cs} for $u_a=\omega_F\times1$ and $u_b=1\times\omega_{\Sigma_g}$ tells us that
\[
\mathrm{TC}(F\times\Sigma_g)=9.
\]
\end{ex}
The latter example extends in particular Corollary \ref{c:betti} and Dranishnikov's result about (manifolds with) hyperbolic groups~\cite{Dr20} to other aspherical manifolds that do not have hyperbolic fundamental groups.

Indeed, Question \ref{que2}(a) asks for a generalisation of the case of negatively curved manifolds whose fundamental groups are hyperbolic and thus have trivial center. In fact, Question \ref{que2}(a) might encompass Question \ref{que1}(b), since an open (folklore) conjecture on aspherical manifolds asserts that $\|M\|=0$ whenever $Z(\pi_1(M))\neq0$. Even more, together with Examples \ref{ex:n-torus} and \ref{ex:product-surface}, one can consider the following much more general and precise question about the role of the center:

\begin{que}\label{que3}
Let $M$ be an aspherical $n$-manifold. Does it hold $\mathrm{TC}(M)=2n+1-\rank(Z(\pi_1(M)))$?
\end{que}

An affirmative answer to Question \ref{que3} holds for several $3$-manifold cases, including 
\begin{itemize}
\item hyperbolic $3$-manifolds $M$, where $\mathrm{TC}(M)=7$ and $\rank(Z(\pi_1(M)))=0$ (by Dranishnikov's result~\cite{Dr20});
\item the $3$-torus $T^3$, where $\mathrm{TC}(T^3)=4$ and $\rank(Z(\pi_1(M)))=3$ (this is part of Example \ref{ex:n-torus});
\item products $M=S^1\times\Sigma_g$, where $\mathrm{TC}(M)=6$ and $\rank(Z(\pi_1(M)))=1$ and
\item non-trivial $S^1$-bundles $M$ over hyperbolic surfaces $\Sigma_g$, where again $\mathrm{TC}(M)=6$ and $\rank(Z(\pi_1(M)))=1$.
\end{itemize}
In the latter two cases, the fundamental group of the $3$-manifold has infinite cyclic center coming from the $S^1$-fiber. The lower bound $6$ for their topological complexity is obtained by the existence of the hyperbolic surface in such $3$-manifolds.  
The upper bound $6$ is given by the product inequality for $S^1\times\Sigma_g$ (see \cite[Theorem 11]{Fa03}) and for the non-trivial $S^1$-bundles over a hyperbolic surface it follows as an application of a result of Grant~\cite[Prop. 3.7]{Gr12}. In fact, Grant gives strong evidence for an affirmative answer to Question \ref{que3} for nilpotent manifolds:

\begin{thm}\cite[Corollary 3.8]{Gr12}\label{t:Grant}
Let $G$ be a finitely generated torsion-free nilpotent group. Then $\mathrm{TC}(G)\leq 2\rank(G)+1 - \rank(Z(G))$.
\end{thm}

In dimension three, Theorem \ref{t:Grant} applies to $T^3$ and to non-trivial $S^1$-bundles $M$ over $T^2$, where $Z(\pi_1(M))\cong\pi_1(S^1)=\Z$ and, thus, $\mathrm{TC}(M)\leq6$ (for nilpotent examples in dimension four see~\cite[Prop. 6.10]{Ne18}; for those $4$-manifolds we have $\mathrm{TC}\leq8$). An interesting remaining case in dimension three (possibly looking towards negative answers to Question \ref{que3}) is that of solvable manifolds. Namely, let $M$ be a mapping torus of an Anosov diffeomorphism of $T^2$. Such $M$ is aspherical, $Z(\pi_1(M))=0$ and there is no atoroidal cohomology class in $M$. In particular, Theorem \ref{t:main} cannot be used to show that $\mathrm{TC}(M)\geq6$. We know, however, that  $\mathrm{TC}(M)\geq5$ by~\cite[Theorem 6.6]{Me21} since $M$ does not admit maps of non-zero degree by direct products~\cite[Theorem 1]{KN13} (note that $\mathrm{TC}\geq5$ holds for nilpotent $3$-manifolds as well for the same reason).

Finally, concerning Question \ref{que2}(b),  in the non-aspherical setting we have the bound $\mathrm{TC}(\#_m(S^2\times S^1))\leq5$, as we already explained in Example \ref{ex:productsspheres}. Due to the small range of values of the topological complexity in dimension three, the following special case of Question \ref{que2} is intriguing: 

\begin{que}\label{que4}
Let $M$ be an aspherical $3$-manifold which is not a rational homology sphere. What is the value of $\mathrm{TC}(M\#M)$?
\end{que}

By~\cite[Theorem 1]{KN13}, such $M\#M$ does not admit maps of non-zero degree from direct products, hence by~\cite[Theorem 6.6]{Me21} we obtain that $5\leq\mathrm{TC}(M\#M)\leq7$. As we mentioned above, quite often (e.g., when $M$ is a hyperbolic $3$-manifold or a circle bundle over a hyperbolic surface), we obtain $\mathrm{TC}(M\# M)\geq6$, because of the existence of an atoroidal class in $H^2(M)$. Hence, Question \ref{que4} becomes even more interesting when there is no such atoroidal class, i.e., when $M$ is the $3$-tours $T^3$, a non-trivial circle bundle over $T^2$ or a mapping torus of an Anosov diffeomorphism of $T^2$.

\bibliographystyle{amsalpha}

\end{document}